\documentclass[12pt]{amsart}

\usepackage{amsmath}
\usepackage{amssymb}
\usepackage{array}
\usepackage{enumitem}
\usepackage{url}
\usepackage{verbatim} 

\theoremstyle{plain}
\newtheorem{theorem}{Theorem}[section]
\newtheorem{lemma}[theorem]{Lemma}
\newtheorem{proposition}[theorem]{Proposition}
\newtheorem{corollary}[theorem]{Corollary}

\newtheorem{definition}[theorem]{Definition}

\theoremstyle{remark}

\newtheorem*{example}{Example}
\newtheorem*{notation}{Notation}
\newtheorem*{acknowledgment}{Acknowledgment}

\numberwithin{equation}{section}

\newcommand{\A}{\mathbb{A}}

\newcommand{\bB}{\mathbb{B}}
\newcommand{\B}{\mathbb{B}}

\newcommand{\K}{\mathbb{K}}

\newcommand{\R}{\mathbb{R}}
\newcommand{\C}{\mathbb{C}}
\newcommand{\N}{\mathbb{N}}

\newcommand{\Z}{\mathbb{Z}}

\newcommand{\sss}{\mathbf{s}}
\newcommand{\rrr}{\mathbf{r}}
\newcommand{\ttt}{\mathbf{t}}
\newcommand{\vvv}{\mathbf{v}}

\newcommand{\Aut}{\mathrm{Aut}}

\newcommand{\id}{\mathrm{id}}

\newcommand{\SJ}{{\mathrm{SJ}}}
\newcommand{\tr}{{\mathrm{tr}}}

\newcommand{\eps}{\varepsilon}

\newcommand{\inv}{^{-1}}
\newcommand{\msk}{\medskip}
\newcommand{\ssk}{\smallskip}
\newcommand{\nin}{\noindent}

\newcommand{\CC}{{\mathcal C}}

\begin{document}

\title[Simplicial differential calculus]{Simplicial differential calculus, divided
differences, and construction of Weil functors}

\author{Wolfgang Bertram}

\address{Institut \'{E}lie Cartan Nancy \\
Nancy-Universit\'{e}, CNRS, INRIA \\
Boulevard des Aiguillettes, B.P. 239 \\
F-54506 Vand\oe{}uvre-l\`{e}s-Nancy, France}

\email{\url{bertram@iecn.u-nancy.fr}}

\subjclass[2000]{ 
13B02, 
13F20, 
13N99, 
14B10, 
39A12, 
58A20, 
58A32} 

\keywords{divided differences, differential calculus, jet functor, scalar extension, Weil functor,
Taylor expansion}

\begin{abstract} 
We define a {\em simplicial differential calculus} by
generalizing {\em divided differences} from the case of curves to the case of 
general maps, defined on general topological vector spaces, or even on modules over
a topological ring $\K$. 
This calculus has the advantage
that the number of evaluation points grows linearly with the degree, and not
exponentially as in the classical, ``cubic'' approach. In particular,
it is better adapted to the case of {\em positive characteristic}, where
it permits to define Weil functors corresponding to scalar extension
from $\K$ to truncated polynomial rings $\K[X]/(X^{k+1})$.
\end{abstract}

\maketitle

\section*{Introduction}

When one tries to develop differential calculus in positive characteristic, a major
problem arises from the fact that the Taylor expansion of a function $f$ involves
a factor $\frac{1}{k!}$ in front of the differential $d^k f(x)$.
In the present work, we define a version
of differential calculus, called {\em simplicial differential calculus}, that allows one to avoid this factor.
The methods are completely general and should be of interest also in the case of
characteristic zero since they point a way to reduce the growth of the number of 
variables from an exponential one, arising in the usual ``cubic'' differential calculus, to a linear one. 
Moreover, we hope that they will build a bridge between differential geometry and algebraic
geometry since they show how to ``embed'' infinitesimal methods used there, based on
``simplicial'' ring extensions, into ordinary cubic differential calculus.  

\ssk
Let us explain the problem first by looking at functions of {\em one} variable (i.e., curves), before
coming to the general case. If $f:I \to W$ is such a function, say of class $\CC^2$, the second
differential at $x$ can be obtained as a limit
$$
f''(x) = \lim_{s,t \to 0} \frac{f(x+t+s) - f(x+t) - f(x+s) + f(x)}{st} \,  .
$$
This formula arises simply from iterating the formula for the first differential.
There are similar formulae for higher differentials $d^k f(x)$;
at each stage the number of points where evaluation of $f$ takes place is doubled, so
that $2^k$ points are involved. This corresponds to the vertices of a hypercube, and
therefore we call this version of differential calculus
``cubic''. A systematic generalization of this calculus is the general differential calculus
developed in \cite{BGN04}, where a characteristic feature is that we look at higher order
difference quotient maps $f^{[k]}$ involving evaluation of $f$ at $2^k$ generically
pairwise different points, and in \cite{Be08} we followed this line of thought by investigating
the differential geometry of {\em higher order tangent maps $T^k f$}. 
The advantage of this calculus is its easy inductive definition; 
the drawback is the exponential growth of variables.

\ssk Now, in the case of curves,  there is another formula for the second differential:
$$
\frac{1}{2} f''(x) =  \lim_{a,b,c \to x} \Big(
\frac{f(a)}{(a-b)(a-c)} + \frac{f(b)}{(b-a)(b-c)} + \frac{f(c)}{(c-a)(c-b)} \Big)
$$
It involves evaluation only at $3$ points, and it has the advantage of automatically producing
the second derivative term of the Taylor expansion, so that we can write the expansion
without having to divide by factorials.
This generalizes to any order: define
{\em divided differences} by the formula (see Chapter 7 in \cite{BGN04},
 \cite{Sch} or  \cite{Rob}; see also the Wikepedia-page on ``divided differences''):
$$
[t_1,\ldots,t_{k+1};f] := \sum_{j=1}^{k+1}
{ f(t_j) \over \prod_{i \not= j} (t_j - t_i) }
$$
where $t_i \not= t_j$.
If $f$ is $\CC^k$ in the usual sense (say, over $\R$ or $\C$), then the divided differences
 admit a {\em continuous} extension to a map 
defined on $I^{k+1}$  (including   the ``singular set'', where some of the
$t_i$'s coincide), and the $k$-th derivative $f^{(k)}(t)$ is obtained as a
``diagonal value'' of this extended function via
$$
\frac{1}{k!}  f^{(k)}(t)= [t,\ldots,t;f]. 
$$
Since evaluation of $f$ only at $k+1$ points is used, we call this definition
of higher order differentials {\em simplicial}. Geometrically, the factor $k!$ represents the
ratio between the volume of the standard hypercube and the standard 
simplex.

\ssk
Next let us look at functions of several (finitely or infinitely many) variables,
say $f:U \to W$ defined on an open part $U$ of some topological vector space $V$.
As shown in \cite{BGN04}, the ``cubic'' calculus generalizes very well to this
framework. However,
to our knowledge, a reasonable ``simplicial'' theory, 
generalizing the divided differences, has so far not been 
developed (see below for some remarks on related literature).
  By ``reasonable'' we mean a calculus
that shares some main features with ``usual'' calculus -- above all, there must be some version of
the chain rule, so that one can define categories like smooth manifolds, bundles, etc.
For this it is not enough to look simply at curves $\gamma$ and to consider
divided differences of the function $f \circ \gamma$; rather, these should appear as
certain special values of the general simplicial theory.
We propose the following general definition of generalized divided differences
(Section 1.2):
$$
f^{\langle k \rangle}(v_0,\ldots,v_k;s_0,\ldots,s_k)  := 
\frac{f(v_0)}{(s_0-s_1)\cdots (s_0-s_k)} +
\frac{f\bigl(v_0+(s_1 - s_0)v_1\bigr)}{(s_1-s_0)(s_1-s_2)\cdots (s_1-s_k)}
$$
$$
 \quad \quad  + \ldots  + 
\frac{f\bigl(v_0+(s_k - s_0)v_1 + \ldots + (s_k-s_0) \cdots (s_k - s_{k-1}) 
v_k \bigr)}{
(s_k-s_0) \cdots (s_k - s_{k-1})}
$$
Here, $\vvv:=(v_0,\ldots,v_k)$ is a $k+1$-tuple of ``space variables'',
and $\sss:=(s_0,\ldots,s_k)$ a $k+1$-tuple of ``time (scalar) variables'';
hence the the number of variables grows linearly with $k$.
If $v_0=x$, $s_0=0$, $v_1=h$ and all other $v_j=0$, then we are back in the case of divided
differences of $f \circ \gamma$ for the curve $\gamma(t)=x+th$.
We will say that $f$ is {\em  $k$ times simplicially differentiable}, or {\em of class
$\CC^{\langle k \rangle}$}, if these generalized divided differences extend to  {\em continuous} maps defined
also for singular scalar values (that is, for $\sss$ such that not all
$s_i - s_j$ are invertible, in particular, to $\sss = (0,\ldots,0)$).

\ssk
The main motivation for this definition is that there is indeed a version of the chain rule,
and that the corresponding theory of manifolds and their bundles permits to define
{\em jet bundles} also in positive characteristic; this seems to be indeed the correct
framework for generalizing the theory of {\em Weil functors} to arbitrary 
characteristic 
(see \cite{KMS}, Chapter VIII for an account on the real theory).\footnote{Note, however, that
Thm 35.5 in loc.\ cit.\ contains an error:   not all finite-dimensional
quotients of polynomial algebras are Weil algebras (this is one of the points of the present
work). A corrected version of this claim can be found in Section 1.5 of \cite{Kolar}.}
To state the chain rule, we define for any $\sss$, the {\em simplicial $\sss$-extension}
to be the vector 
$$
\SJ^{(\sss)} f (\vvv):=
\Bigl( f(v_0), f^{\langle 1 \rangle}(v_0,v_1;s_0,s_1),\ldots,f^{\langle k \rangle}(\vvv;\sss) \Bigr) \, ,
$$
so that $\SJ^{(\sss)} f:\SJ^{(\sss)}U \to W^{k+1}$ is a map from an open part
$\SJ^{(\sss)}U$ of $V^{k+1}$ to $W^{k+1}$.
Then the chain rule (Theorem \ref{SimplicialChainRule}) 
says that $\SJ^{(\sss)} (g \circ f) = \SJ^{(\sss)} g \circ \SJ^{(\sss)} f$, i.e.,
 $\SJ^{(\sss)}$ {\em is a covariant functor}. In particular, for
$\sss = (0,\ldots,0)$ this really is a true generalization of the classical chain rule. 
Closely related to this is a result  (Theorem \ref{LimitedTheorem}
  and Corollary \ref{LimitedCorollary}) characterizing 
$\CC^{\langle k \rangle}$-maps as the maps satisfying a certain ``limited expansion'', which
contains as a special case a version of the Taylor expansion involving only
the ``simplicial differentials'' $\SJ^{(0,\ldots,0)}f$, without division by factorials.

\ssk
The chain rule leads directly to the algebraic viewpoint of {\em scalar extension} (Chapter 2), and to the
construction of {\em Weil functors} (Chapter 3).
Here we take advantage of the generality of our framework, allowing to take for $\K$
a commutative topological base {\em ring} -- all definitions and results mentioned so far
make sense in this generality. Now combine this with
the basic observation from the theory of Weil functors: 
applying a covariant, product preserving functor like
$F:=\SJ^{(\sss)}$ to the base ring $\K$ with its structure maps $a$ (addition) and $m$
(multiplication), we get again a ring $(F\K,Fa,Fm)$.
The ring $\SJ^{(\sss)}\K$ thus obtained
is never a field (even if $\K$ is), but it is still a well-behaved
commutative topological ring, and therefore we can speak of smooth maps over this ring.
We prove  (Theorem \ref{SimpScalExtTheorem}): {\em If $f$ is of class $\CC^{<k+m>}$ over $\K$, then
$\SJ^{(\sss)}f$ is of class $\CC^{<m>}$  over $\SJ^{(\sss)}\K$}, and we determine explicitly the
structure of $\SJ^{(\sss)}\K$ (Lemma \ref{SimpScalExtLemma}): 
{\em The ring $\SJ^{(\sss)}\K$ is naturally isomorphic to
the truncated polynomial ring}
$$
\B^\sss:=\K[X]/\bigl( X(X-(s_1-s_0)) \cdot \ldots \cdot (X-(s_k - s_0)) \bigr) \, .
$$
Putting these two results together, for $\sss = 0$, we may say that {\em
$\SJ^{(0,\ldots,0)}$ is the functor of scalar extension from $\K$ to
$\K[X]/(X^{k+1})$}.
These results have multiple applications:
on the one hand, as long as $\sss$ is non-singular (i.e., if  $s_i - s_j$ is invertible for $i \not= j$),
they reduce the complicated structure of finite differences to the 
better accessible structure of rings. On the other hand, the functors $\SJ^{(\sss)}$
carry over to the category of manifolds. To be precise, if $\sss \not= (0,\ldots,0)$,
then $\SJ^{(\sss)}$ is a functor in the category of {\em manifolds with atlas}
(Theorem \ref{FunctorTheorem}), and if $\sss = (0,\ldots,0)$, then
difference calculus contracts to ``local (i.e., support-decreasing) calculus'', 
hence leads to differential geometry: in this case
the functor of scalar extension $\SJ^{(\sss)}$ carries over to the category of manifolds, defining
for each $\K$-manifold $M$ a bundle $\SJ^kM:=\SJ^{(0,\ldots,0)} M$ over $M$ which is
independent of the atlas of $M$.
This is precisely the version of the jet functor that works well in any characteristic,
and it now makes sense to consider $\SJ^k M$ as a manifold {\em defined over
the ring $\SJ^k \K = \K[X]/(X^{k+1})$} (Theorem \ref{WeilFunctorTheorem}).

\ssk
A second main topic of the present work is to investigate the
 relation between ``cubic'' and ``simplicial'' calculus.
Indeed, the point of view of Weil functors has been already investigated in the
``cubic'' framework (\cite{Be08}), where we have  observed that this
framework leads to some loss of information in
 the case of {\em positive} characteristic.\footnote{{\em cf}.\ the note in loc.\ cit., p.\  13;
the theory from \cite{Be08} works in arbitrary characteristic, mainly because we use there
the ``second order Taylor formula'' from \cite{BGN04}, which is already simplicial in nature; but the
theory itself is not simplicial.}
We recall in Section 1.1 the ``cubic'' $\CC^{[k]}$-concept from \cite{BGN04}, and we prove
(Theorem \ref{cubictosimplicialTh}):
{\em ``Cubic implies simplicial'': if $f$ is $\CC^{[k]}$, then $f$ is $\CC^{\langle k \rangle}$.
Moreover, there is an ``embedding'' of the simplicial divided differences into the
cubic higher order difference quotients.} The latter are far too complicated  to allow for an explicit,
``closed'' formula which would be comparable to the simplicial formula given above; all the more it is
appreciable that the point of view of scalar extension works also on the the cubic level:
we define inductively a family of rings $\A^{\ttt}$ (where now $\ttt \in \K^{2^k - 1}$, and
the $\K$-dimension of $\A^\ttt$ is $2^k$) and show (Theorem \ref{scalextTh3}):
{\em The cubic extended tangent functor $T^{(\ttt)}$ can be interpreted as the scalar
extension functor from $\K$ to $\A^\ttt$.}
In particular, for $\ttt = 0$, we get the higher order tangent functors $T^k$ considered
in \cite{Be08}. 
The embedding of simplicial divided differences into the cubic theory then translates
into algebra (Theorem \ref{RingEmbeddingTheorem}):
{\em There is an embedding of algebras $\B^{\sss} \to \A^{\ttt}$ (where $\ttt = \ttt(\sss)$ depends
on $\sss$).}
Correspondingly, if $f$ is $\CC^{[k]}$, the ``simplicial Weil functors'' can be embedded into a family of
``cubic Weil functors'' (Theorem \ref{WeilFunctorTheorem}). 
This embedding is ``off-diagonal'' (i.e., ``most'' components of $\ttt$ are zero, but some are not),
and  has a more subtle structure than
the ``diagonal'' embedding used in \cite{Be08}.

\ssk
Let us add some (possibly incomplete) remarks on related literature. 
The definition of ``$r$-th order difference factorizer''
 in Section 5.b of \cite{Nel88} may be seen as an attempt to define some kind
 of divided differences for several variables; these objects
 are defined in a similar way as the usual divided differences, leading to the serious
 drawback that they are no longer uniquely defined by $f$ as soon as the space dimension
 is bigger than one. Having introduced them, L.D.\ Nel  decides ``not to study them
in any depth in this paper''.
``Simplicial'' objects appear also in {\em synthetic differential geometry}
(see Section I.18 of \cite{Ko81})  and in algebraic geometry
(see \cite{BM01}); the approach is different, and it
would be interesting to investigate in more detail the relation with the theory developed here.

\ssk
Finally, let us mention some open problems and further topics (see also \cite{Be08b}).
Firstly, we conjecture that the converse of Theorem \ref{cubictosimplicialTh} also holds: {\em 
``simplicially smooth implies cubically smooth'', hence both concepts are equivalent}
(to be more precise, we conjecture that this is true at least if  $\K$ is a field
since that assumption has turned out to be sufficient for a similar result concerning
curves, see \cite{BGN04}, Prop.\ 6.9). 
A proof of this conjecture would imply that the simplicial differential 
is always {\em polynomial} (which is indeed the case under the assumption that $f$
be $\CC^{[k]}$), and should also indicate a procedure how to recover, in a 
natural way, the algebras $\A^\ttt$ from the ``smaller'' algebras $\B^\sss$.
Secondly, 
the simplicial point of view suggests the adaptation of
the ``cubic'' differential geometry and Lie theory from \cite{Be08} to this framework.
In a certain sense, this would amount in the combination of the theory of scalar extensions and
Weil functors with ``simplicial'' concepts present in \cite{White}. 
This is of course a vast topic, which will be taken up elsewhere.

\begin{acknowledgment}
I thank my collegue Alain Genestier for stimulating discussions and for
suggesting to me the correct formula describing the generalized 
divided differences, and the unknown referee for careful reading and helpful suggestions. 
\end{acknowledgment}

\begin{notation}
In the following, the base ring $\K$ will be a {\em unital commutative topological ring with dense unit
group} $\K^\times \subset \K$; all $\K$-modules $V,W$ will be {\em topological
$\K$-modules}, and domains of definition $U$ will be {\em open} (or, more generally,
subsets having a dense interior). The class of continuous maps will be denoted by
$\CC^{[0]}$ or $\CC^{<0>}$. 
For some purely algebraic results in Chapter 2, the topology will not be necessary, and
one might instead use arguments of Zariski-density; we leave
such modifications to the reader. 
\end{notation}

\section{Differential calculi}

\subsection{Cubic differential calculus}

 We recall the basic
definitions of the  ``cubic'' theory
developed in \cite{BGN04} (see also Chapter 1 of \cite{Be08}).

\begin{definition} \label{cubicDef} We 
 say that $f:U \to W$ {\em is of class $\CC^{[1]}$} if the {\em first order
difference quotient map}
$$
(x,v,t) \mapsto \frac{f(x+tv)-f(x)}{t}
$$
extends continuously onto the {\em extended domain}
$$
U^{[1]}:=\{ (x,v,t) \in V \times V \times \K | \,
x \in U, x+tv \in U \},
$$
 i.e., if there exists a continuous map
$f^{[1]}:U^{[1]} \to W$ such that
$f^{[1]}(x,v,t)=\frac{f(x+tv)-f(x)}{t}$ 
whenever $t$ is invertible.
By density of $\K^\times$ in $\K$, the map $f^{[1]}$ is unique if it exists, and so is
the value
$$
df(x)v:=f^{[1]}(x,v,0).
$$
The {\em extended tangent map} is then defined by
\begin{align*}
\hat T f: & U^{[1]} \to W^{[1]}=W \times W \times \K, \cr
 & (x,v,t) \mapsto \hat Tf(x,v,t):=\hat T^{(t)}f(x,v):=\big( f(x), f^{[1]}(x,v,t),t \big) .
\end{align*}
\end{definition}

If $f$ is $\CC^{[1]}$,
the differential $df(x):V \to W$ is continuous and linear, and
$\hat T$ is a functor: $\hat T (g \circ f) = \hat Tg \circ \hat Tf$;
this is equivalent to saying that for each $t \in \K$ we have a functor
$\hat T^{(t)}$, and
for $t=0$ this gives the usual chain rule (see loc.\ cit.\ for the easy proofs).
Moreover, for each $t$, the functor $\hat T^{(t)}$ commutes with direct products:
$\hat T^{(t)}(g \times f)$ is naturally identified with
$\hat T^{(t)} f \times \hat T^{(t)} g$. 

\begin{definition}
The classes $\CC^{[k]}$ are defined by induction:
we say that $f$ is {\em of class $\CC^{[k+1]}$} if it is of class $\CC^{[k]}$
and if $f^{[k]}:U^{[k]} \to W$ is again of class $\CC^{[1]}$,
where $f^{[k]}:=(f^{[k-1]})^{[1]}$.
The {\em higher order extended tangent maps} are defined by
$\hat T^{k+1} f:= \hat T (\hat T^k f)$. 
\end{definition}

Among the higher order differentiation rules proved in \cite{BGN04}, 
{\em Schwarz's Lemma} and the
{\em generalized Taylor expansion} are the most important. Both will be discussed in 
more detail later on. Explicit formulae for the higher order difference quotient maps
tend to be very complicated. For convenience of the reader, we give here the explicit formula
in case $k=2$:
\begin{eqnarray*}
& & f^{[2]}((x,v_1,t_1),(v_2,v_{12},t_{12}),t_{2})   \cr
& & \quad=
{f^{[1]}((x,v_1,t_1)+t_{2} (v_2,v_{12},t_{12})) - f^{[1]}(x,v_1,t_1)
\over t_{2}} \cr
& & \quad  =
{f\bigl( x + t_{2} v_2 + (t_1+t_2 t_{12}) ( v_1 + t_{2} v_{12}) \bigr) 
-f(x + t_{2} v_2)
\over  t_{2} (t_1  + t_2 t_{12}) }
- {f(x+t_1 v_1) - f(x) \over t_1 t_{2} } 
\cr
\end{eqnarray*}
where of course it is assumed that the scalars in the denominator belong to $\K^\times$.
Observe also that the factor $t_{12}$ never stands alone, 
hence in the limit $(t_1,t_2) \to (0,0)$, we obtain a
local (i.e., support-decreasing) operator even if $t_{12}$ does not tend to zero (e.g., for $t_{12}=1$);
in finite dimensions over $\R$, by the classical Peetre Theorem (see \cite{KrM97}),
we thus obtain a differential operator. In the general case,
this observation will be taken up later (Theorem \ref{WeilFunctorTheorem}).
It would be hopeless to try to develop the theory by writing out in this way
the formulae for $f^{[k]}$ in general --  they
 involve, in a fairly complicated way, the values of $f$ at $2^k$ generically different points. 
Here is a first step towards an efficient organization of
 these variables: we group together ``space variables'' on 
one hand and ``time variables'' on the other hand; that is, $f^{[k]}$ contains
$2^k-1$ variables from $\K$ which we may fix and look at the remaining transformation
on the space level:

\begin{definition}
Let $I=\{ 1,\ldots ,n \}$ be the standard $n$-element set and fix a family
$\ttt := (t_J)_{J \subset I, J \not= \emptyset}$ of elements $t_J \in \K$.
In other words, $\ttt \in \K^{2^k-1}$.
If $J= \{ i_1, \ldots, i_k \} $, then instead of
$t_J$ we write also $t_{i_1,\ldots,i_k}$. 
We define the {\em (cubic) $\ttt$-extension $T^{(\ttt)}f$ of $f$} 
to be the partial map of $\hat T^n f$ where the scalar parameters have the fixed value
$\ttt$. Thus, 
for $n=1$, $T^{(t)} f(x,v)=(f(x),f^{[1]}(x,v,t))$ is the map introduced above, and 
for $n=2$ we get with $\ttt = (t_1,t_2,t_{1,2})$

\ssk
$
T^{(\ttt)}f(x,v_1,v_2,v_{1,2}) =
$
$$
\quad \quad \quad 
\Big (f(x),f^{[1]}(x,v_1,t_1),f^{[1]}(x,v_2,t_2),
f^{[2]}((x,v_1,t_1),(v_2,v_{1,2},t_{1,2}),t_2)  \Big).
$$
For general $\ttt$, $T^{(\ttt)} f$ is defined on an open set $T^{(\ttt)} U \subset V^{2^n}$ and 
takes  values in
$W^{2^n}$.
\end{definition}

By induction it follows immediately from the remarks concerning the case $n=1$
that $T^{(\ttt)}$ is a covariant functor preserving direct products.
For $\ttt = 0$, this is the higher order tangent functor denoted by $T^n$ in \cite{Be08}.

\subsection{Simplicial differential calculus}

We will write $k$-tuples of vectors or of scalars in the form
$\vvv :=(v_0,\ldots,v_k) \in V^{k+1}$,
$\sss :=(s_0,\ldots,s_k) \in \K^{k+1}$, and we will say that $\sss$ is {\em non-singular} if,
for $i \not= j$, $s_i - s_j$ is invertible.

\begin{definition}\label{simplDef}
For a map $f:U \to W$ and non-singular $\sss \in \K^{k+1}$,
 we define  {\em (generalized)
divided differences} by 
\begin{eqnarray*}
f^{>k<}(\vvv;\sss )& := &
\frac{f(v_0)}
{\prod_{j=1,\ldots,k}(s_0-s_j)}
+ 
\sum_{i=1}^k
\frac{f\bigl(v_0 + \sum_{j=1}^i \prod_{\ell=0}^{j-1} (s_i - s_\ell) v_j\bigr)}
{\prod_{j=0,\ldots,k \atop j \not= i}(s_i-s_j)} .
\end{eqnarray*}
For convenience, we spell  this formula out explicitly, as follows: $f^{>0<}(v_0;s_0) := f(v_0)$ and
\begin{eqnarray*}
f^{\rangle 1 \langle}(v_0,v_1;s_0,s_1) & = &
\frac{f(v_0)}{s_0-s_1} +
\frac{f(v_0+(s_1 - s_0)v_1)}{s_1-s_0} 
\cr
f^{>2<}(v_0,v_1,v_2;s_0,s_1,s_2) & = &
\frac{f(v_0)}{(s_0-s_1)(s_0-s_2)} +
\frac{f(v_0+(s_1 - s_0)v_1)}{(s_1-s_0)(s_1-s_2)} + \cr
& & \quad \quad \quad \quad \quad \quad
\frac{f(v_0+(s_2 - s_0)v_1 + 
(s_2 - s_1)(s_2-s_0)v_2)}{(s_2-s_0)(s_2-s_1)}
\end{eqnarray*}
and
\begin{eqnarray*}
f^{>k<}(\vvv;\sss )& := &
\frac{f(v_0)}{(s_0-s_1)\cdot \ldots \cdot (s_0-s_k)} +
\frac{f\bigl(v_0+(s_1 - s_0)v_1\bigr)}{(s_1-s_0)(s_1-s_2) \cdot \ldots \cdot (s_1 - s_k)} + \cr
& & \quad 
 \frac{f\bigl(v_0+(s_2 - s_0)v_1 + 
(s_2 - s_1)(s_2-s_0)v_2\bigr)}{(s_2-s_0)(s_2-s_1)(s_2 - s_3) \cdot \ldots \cdot (s_2 - s_k)}+ \ldots +  \cr
& & 
\frac{f \bigl( v_0 + (s_k - s_0) v_1 + \ldots + (s_k - s_{k-1})(s_k - s_{k-2}) \cdot \ldots \cdot (s_k - s_0) v_k 
\bigr)}
{(s_k-s_0)(s_k - s_1) \cdot \ldots \cdot (s_k - s_{k-1}) } 
\end{eqnarray*}
We say that $f$ {\em is of class $\CC^{\langle k \rangle}$}, or {\em $k$ times
continuously simplicially differentiable}, if $f^{\rangle \ell \langle}$ extends
continuously to singular values of $\sss$, for all $\ell =1,\ldots,k$. This means that  there are
continuous maps $f^{<\ell>}:U^{<\ell>} \to W$, where
$$
U^{<\ell>}:=\big\{ (\vvv,\sss) \in V^\ell \times \K^\ell | \, 
v_0 \in U, \quad \forall i=1,\ldots,\ell : \,
v_0+\sum_{j=1}^i \prod_{m=1}^j (s_i - s_m) v_j \in U \big\},
$$
such that, whenever $(\vvv,\sss) \in U^{<\ell>}$ and $\sss$ is non-singular,
$$
f^{\rangle \ell \langle}(\vvv;\sss) = f^{<\ell>}(\vvv;\sss).
$$
The map $f^{<\ell>}$ will be called the {\em extended
divided difference map}.
Note that, by density of $\K^\times$ in $\K$, the extension $f^{<\ell>}$
is unique (if it exists), and hence in particular the value
$f^{<\ell>}(\vvv ; {\bf 0})$,
called the {\em $\ell$-th order simplicial differential},
is uniquely determined.
\end{definition}

One may observe that
$f^{\langle k \rangle}(\vvv;s_0,\ldots,s_k) = f^{\langle k \rangle}(\vvv;s_0 - t,\ldots,s_k-t)$ for all $t \in \K$ since
only differences of scalar values appear in the definition; in particular, we may choose $t=s_0$,
so that for many purposes one may assume that $s_0=0$. 
It is clear that $f$ is $\CC^{[1]}$ if and only if $f$ is $\CC^{\langle 1 \rangle}$, since
\begin{equation}\label{Efeins}
f^{[1]}(x,v,t)=f^{\langle 1 \rangle}(x,v;0,t), \quad \quad
f^{\langle 1 \rangle}(v_0,v_1;s,t)=f^{[1]}(v_0,v_1,t-s) .
\end{equation}
For $k>1$, it is less easy to compare both concepts. 
In order to attack this problem, we start by proving a recursion formula:

\begin{lemma}\label{RecursionLemma}
The following recursion formula holds: for non-singular $\sss$,
$$
f^{>k+1<}(v_0,\ldots,v_{k+1}; s_0,\ldots,s_{k+1})=
\frac{1}{s_k - s_{k+1}} \Bigl(
f^{>k<}(v_0,\ldots,v_k;s_0,\ldots,s_k) - \quad \quad { }
$$
$$
\quad \quad 
f^{>k<}\bigl(v_0,v_1,\ldots,v_{k-1},v_{k}+(s_{k+1}-s_k)v_{k+1};
s_1,\ldots,s_{k-1},s_{k+1} \bigr) \Bigr)
$$
\end{lemma}

\begin{proof} We are going to compute the right-hand side term of this equation. To this end,
observe that, in the definition of $f^{>k<}(\vvv;\sss)$, the values of $f$ at $k+1$ (generically 
pairwise different) points occur, where the $j$-th point depends only on
$(v_1,\ldots,v_j;s_0,\ldots,s_j)$. 
Hence, for $j=0,\ldots,k-1$, these points of evaluation are the same for both terms
of which the difference is taken. 
Using the algebraic identity
$$
\frac{1}{a-c} \bigl( \frac{1}{b-a} - \frac{1}{b-c} \bigr)=
\frac{1}{(b-a)(b-c)} ,
$$
we get for the difference of two such terms
$$
\frac{1}{s_k - s_{k+1}} 
\Bigl(
\frac{ f\bigl( v_0 + (s_j - s_0) v_1 + \ldots + (s_j - s_{j-1}) \ldots (s_j - s_0) v_j \bigr)}
{(s_j - s_k) \prod_{i = 0,\ldots,k-1 \atop i \not= j} (s_j - s_i) }  \, - \quad \quad \quad \quad \quad \quad 
$$
\begin{eqnarray*}
 & &
\quad \quad \quad \quad \quad \quad 
\frac{ f\bigl( v_0 + (s_j - s_0) v_1 + \ldots + (s_j - s_{j-1}) \ldots (s_j - s_0) v_j \bigr)}
{(s_j - s_{k+1}) \prod_{i = 0,\ldots,k-1 \atop i \not= j} (s_j - s_i) }
\Bigr)
\cr
&=& \frac{1}{s_k - s_{k+1}} \bigl( \frac{1}{s_j - s_{k}} - \frac{1}{s_j - s_{k+1}} \bigr)
\frac{ f\bigl( v_0 + (s_j - s_0) v_1 + \ldots + (s_j - s_{j-1}) \ldots (s_j - s_0) v_j \bigr)}
{\prod_{i = 0,\ldots,k-1 \atop i \not= j}(s_j - s_i) }
 \cr
 &=&
 \frac{ f\bigl( v_0 + (s_j- s_0) v_1 + \ldots + (s_j - s_{j-1}) \ldots (s_j - s_0) v_j \bigr)}
{\prod_{i = 0,\ldots,k+1 \atop i \not= j}(s_j - s_i) }
\end{eqnarray*}
which is exactly the $j$-th term appearing in the definition of 
$f^{>k+1<}$.  
It remains to show that the difference of the $k$-th terms leads exactly to the last
two terms in the definition of 
$f^{>k+1<}$.  
Now, from
$$
 \frac{1}{s_k - s_{k+1}} 
 \frac{ f\bigl( v_0 + (s_k - s_0) v_1 + \ldots + (s_k - s_{k-1}) \ldots (s_k - s_0) v_k\bigr)}
{ \prod_{i = 0,\ldots,k-1} (s_k - s_i) }
$$
we readily obtain the $k$-th term, and from 
$$
- 
 \frac{1}{s_k - s_{k+1}} 
 \frac{ f\bigl( v_0 + (s_{k+1} - s_0) v_1 + \ldots + (s_{k+1} - s_{k-1}) \ldots (s_{k+1} - s_0)
 ( v_k + (s_{k+1} - s_k) v_{k+1} ) \bigr)}
{ \prod_{i = 0,\ldots,k-1} (s_{k+1} - s_i) }
$$
we get the last term.
\end{proof}

\begin{theorem}\label{cubictosimplicialTh}
If $f$ is of class $\CC^{[k]}$, then $f$ is of class $\CC^{\langle k \rangle}$. Moreover,
$f^{\langle j \rangle}$ (with $j=0,\ldots,k$)  is  then of class $\CC^{[k-j]}$, and
the following relation holds for all $\sss$: 
$$
- f^{\langle k \rangle}(\vvv;\sss) =
$$
$$
(f^{\langle k-1 \rangle})^{[1]}\Bigl( \bigl( v_0,\ldots,v_{k-1};s_0,\ldots,s_{k-1} \bigr),
\bigl( 0,\ldots,0,v_k;0,\ldots,0,1 \bigr),s_k - s_{k-1} \Bigr).
$$
The simplicial differential quotient maps can be embedded into the cubic ones in the
sense that there exist $\CC^{[\infty]}$- (in fact, affine continuous) 
maps $g_k: U^{\langle k \rangle} \to U^{[k]}$ such that
$$
f^{\langle k \rangle}(\vvv;\sss) = \pm  f^{[k]}(g_k(\vvv;\sss)) .
$$
For $k \in \{ 1,2,3 \}$ we have the following explicit formulae for these embeddings:
\begin{eqnarray*}
f^{\langle 1 \rangle}(\vvv;\sss) & = & f^{[1]}\bigl(v_0,v_1,s_1 - s_0 \bigr)
\cr
f^{<2>}(\vvv;\sss) &=& - f^{[2]}\bigl((v_0,v_1,s_1 - s_0 ),(0,v_2,1), s_2 - s_1\bigl)
\cr
f^{<3>}(\vvv;\sss) &=& f^{[3]}\Bigl( \bigl((v_0,v_1,s_1 - s_0 ),(0,v_2,1), s_2 - s_1\bigl),
\bigl( (0,0,0),(0,v_3,0),1 \bigr), s_3 - s_2 \Bigr)
\end{eqnarray*}
\end{theorem}

\begin{proof} All claims are proved by induction, the case $k=1$ being trivial thanks to
Equation (\ref{Efeins}).
We write the recursion formula from the preceding lemma as 
$$
f^{>k+1<}(v_0,\ldots,v_{k+1}; s_0,\ldots,s_{k+1})=
\frac{1}{s_k - s_{k+1}} \Bigl(
f^{>k<}(v_0,\ldots,v_k;s_0,\ldots,s_k) - \quad \quad { }
$$
$$
\quad \quad 
f^{>k<}\bigl(v_0,v_1,\ldots,v_{k-1},v_{k}+(s_{k+1}-s_k)v_{k+1};
s_1,\ldots,s_{k-1},s_k + (s_{k+1} - s_k) \bigr) \Bigr) 
$$
which has the form of a first order difference quotient,  equal to
$$
-(f^{\langle k \rangle})^{[1]}\Bigl( \bigl( v_0,\ldots,v_{k};s_0,\ldots,s_{k} \bigr),
\bigl( 0,\ldots,0,v_{k+1};0,\ldots,0,1 \bigr),s_{k+1} - s_k \Bigr)
$$
for non-singular $\sss$.
Now assume we have proved the claims of the theorem for order $k$, and let
$f$ be a map of class $\CC^{[k+1]}$.
By induction, $f^{\langle k \rangle}$ is thus of class $\CC^{[1]}$, and hence the right hand side term
from the formula extends to a continuous map of $(\vvv;\sss)$ on the extended domain.
Thus $f^{>k+1<}(\vvv;\sss)$ indeed admits a
continuous extension onto the extended domain, given by the right hand side term.
This proves that $f$ is $\CC^{<k+1>}$, and that the formula for
$f^{<k+1>}(\vvv;\sss)$ from the claim holds.
Note that, by density, this formula holds for all $\sss$ (including singular values).
Moreover, it shows that  $f^{<k+1>}$ is embedded into $f^{[k+1]}$ by a continuous affine map,
and hence all maps $f^{\langle j \rangle}$, being composition of $\CC^{[k+1-j]}$-maps, are again
$\CC^{[k+1-j]}$, by the chain rule.
\end{proof}

As mentioned in the introduction, we conjecture that 
the concepts $\CC^{[k]}$ and $\CC^{\langle k \rangle}$ are equivalent; however,
the proof of the converse of the statement from the theorem is likely to be considerably
more complicated.
For the purposes of the present work, this converse is not needed, as it is clearer and
more instructive to develop the $\CC^{\langle k \rangle}$-theory independently from the
$\CC^{[k]}$-theory, before comparing both approaches.
Thus,
in the following, we develop the basic simplicial theory.
First of all, it is clear that $f^{\langle 1 \rangle}(v_0,v_1;0,0) = df(v_0)v_1$ is the usual first differential.
We will explain now how higher coefficients like
$f^{<2>}(v_0,v_1,v_2;0,0,0)$ are related to second and higher differentials.

\begin{theorem} \label{LimitedTheorem}
Assume $f:U \to W$ is of class $\CC^{\langle k \rangle}$.
Then  the following
``limited expansions'' hold: for all $\sss$, 
\begin{eqnarray*}
f \bigl(v_0 + (s_1 - s_0)v_1) \bigr) & =
& f(v_0) + (s_1 - s_0) f^{\langle 1 \rangle} (v_0,v_1; s_0,s_1)
\cr
f \bigl(v_0 + (s_2 - s_0) (v_1 + (s_2 - s_1)v_2) \bigr)& =&
f(v_0) + (s_2 - s_0) f^{\langle 1 \rangle}(v_0,v_1;s_0,s_1)  + \cr
& & \quad 
(s_2 - s_1)(s_2-s_0) f^{<2>} (v_0,v_1,v_2;s_0,s_1,s_2)
\cr
& \vdots & 
\cr
f \bigl(v_0 + \sum_{j=1}^k \prod_{\ell=0}^{j-1} (s_k - s_\ell) v_j\bigr) & =&
f(v_0) + \sum_{j=1}^k \prod_{\ell=0}^{j-1} (s_k - s_\ell) f^{\langle j \rangle} (v_0,\ldots,v_j;s_0,\ldots,s_j) 
\end{eqnarray*}
In particular, choosing $s_0 = s_1 = \ldots = s_{k-1} = 0$ and $s_k = t$, we get
\begin{eqnarray*}
f(v_0 + t v_1 + t^2 v_2 + \ldots + t^k v_k)& = &
f(v_0) + t f^{\langle 1 \rangle}(v_0,v_1;0,0) + t^2 f^{<2>}(v_0,v_1,v_2;0,0,0) 
\cr
& & \quad \quad \quad \quad + \ldots +
t^k f^{\langle k \rangle}(\vvv;0,\ldots,t) ,
\end{eqnarray*}
which, for $v_2 = \ldots = v_k = 0$ and $v_1=:h$ gives the {\em radial Taylor expansion}
\begin{eqnarray*}
f(v_0 + t h )& = &
f(v_0) + t f^{\langle 1 \rangle}(v_0,h;0,0) + t^2 f^{<2>}(v_0,h,0;0,0,0) 
\cr
& & \quad \quad \quad \quad + \ldots +
t^k f^{\langle k \rangle}(v_0,h,\ldots,0;0,\ldots,t) .
\end{eqnarray*}
\end{theorem}

\begin{proof} 
The claim is proved by induction. The computation can be seen as multivariable analog of the
proof of the generalized Taylor expansion from \cite{BGN04}, Th.\ 5.1, consisting of a repeated
application of the relation $f(x+tv)=f(x)+t f^{[1]}(x,v,t)$.
For $k=1$ we have
\begin{eqnarray*}
f(v_0 + (s_1 - s_0)v_1)& =& f(v_0) + (s_1 - s_0) f^{[1]}(v_0,v_1,s_1-s_0)
\cr
& =&
f(v_0) + (s_1 - s_0) f^{\langle 1 \rangle}(v_0,v_1;s_0,s_1).
\end{eqnarray*}
For $k=2$, replace in the preceding equation $s_1$ by $s_2$ and $v_1$ by $v_1+(s_2 - s_1)v_2$,
\begin{eqnarray*}
f \bigl(v_0 + (s_2 - s_0)( v_1 + (s_2 - s_1)v_2) \bigr)& =&
 f(v_0) + (s_2 - s_0) f^{\langle 1 \rangle}(v_0,v_1 + (s_2 - s_1)v_2;s_0,s_2)
\cr
& =&
f(v_0) + (s_2 - s_0) \bigl(  f^{\langle 1 \rangle}(v_0,v_1;s_0,s_1)+ 
\cr & &
\quad \quad \quad  (s_2-s_1) 
 f^{<2>}(v_0,v_1,v_2;s_0,s_1,s_2) \bigr)
\end{eqnarray*}
where  for the last equality 
we used the recursion formula  (in its form valid on the extended domain, given in
Theorem \ref{cubictosimplicialTh}).
For $k=3$, 
we replace again $s_2$ by $s_3$ and $v_2$ by $v_2 + (s_3 - s_2)v_3$,
and proceed in the same way, and so on. The remaining statements are immediate consequences.
\end{proof}

\begin{corollary}\label{DiffCor}
Assume $f:U \to W$ is of class $\CC^{[k]}$ and denote by
$$
f(x+th)=f(x) + ta_1(x,h) + t^2 a_2(x,h) + \ldots + t^k a_k(x,h) + t^k R_{k}(x,h,t)
$$
the radial Taylor expansion of $f$ at $x$ from \cite{BGN04}, Theorem 5.1.
Then this expansion coincides with the one given in the preceding theorem, that is,
$$
a_j(x,h)= f^{\langle j \rangle}( x,h,0,\ldots,0;0,\ldots,0 ).
$$
In particular, the maps $h \mapsto  f^{\langle j \rangle}( x,h,0,\ldots,0; 0,\ldots,0 )$ are polynomial.
If $2$ is invertible in $\K$, then 
$$
f^{<2>}(v_0,v_1,v_2;0, 0,0) = df(v_0) v_2 + \frac{1}{2} d^2 f(v_0)(v_1,v_1),
$$
and if $2$ and  $3$ are invertible in $\K$, then
$$
f^{<3>}(v_0,v_1,v_2,v_3 ; 0,0, 0,0) 
= df(v_0) v_3 +  d^2 f(v_0)(v_1,v_2) +  \frac{1}{6} d^3 f(v_0)(v_1,v_1,v_1),
$$
and if $2,\ldots,k$ are invertible in $\K$, then $f^{\langle k \rangle}(\vvv;{\mathbf 0})$ is polynomial
in $\vvv$, and
$$
f^{\langle k \rangle}(v_0,v_1,0,\ldots,0; 0,\ldots,0) = \frac{1}{k!} d^k f(v_0) (v_1,\ldots,v_1) .
$$
 \end{corollary}
 
\begin{proof}
The first claim follows from uniqueness of the radial Taylor expansion (see \cite{BGN04}, Lemma  5.2).
It has been shown in \cite{BGN04}, Theorem 5.6, that the coefficients $a_j(x,h)$ are polynomial
mappings in $h$, hence $ f^{\langle j \rangle}( x,h,0,\ldots,0; 0,\ldots,0 )$ is polynomial in $h$
Note that, as shown in \cite{BGN04}, 
$a_2(x,h)=\frac{1}{2} d^2 f(x)(h,h)$ and
$a_3(x,h)=\frac{1}{6} d^2 f(x)(h,h,h)$ (if $2$, resp.\ $6$, are invertible).
For the remaining statements, 
we use again the radial Taylor expansion for $f(x+th)$ and let
$h=v_1 + t v_2 + t^2 v_3$,
$$
f(v_0 + t (v_1 + t v_2 + t^2 v_3)) =
f(v_0) + t df(v_0) h
+ \frac{t^2}{2} d^2 f(v_0) (h,h) + \frac{t^3}{6} d^3f(v_0)(h,h,h) + t^3 R_3,
$$
use multilinearity and symmetry of $d^2 f(v_0)$ and of $d^3 f(v_0)$ and compare terms
according to powers of $t$ with the limited expansion from the theorem; 
uniqueness of these terms leads to the two equalities concerning $f^{<2>}$ and $f^{<3>}$.
Clearly, this procedure can be applied at any order, leading to an explicit and polynomial
formula for $f^{\langle k \rangle}(\vvv;{\mathbf 0})$ (we leave it to the reader to work out the explicit 
combinatorial formula involving all higher differentials
$d^jf(v_0)$, $j=1,\ldots,k$; it  has the same structure as the formula for the highest 
component in $J^k f(v_0)$ given in \cite{Be08}, Theorem 8.6).
\end{proof}

As mentioned above, we conjecture that the converse of Theorem \ref{cubictosimplicialTh} 
holds.
 It should be a major step towards the proof of the conjecture
to prove, if $f$ is assumed $\CC^{\langle k \rangle}$, that $f^{\langle k \rangle}(\vvv;{\mathbf 0})$
is always polynomial in $\vvv$.  In this context,  note that
from the expression of $f^{<3>}$, we should indeed be able to recover the 
second differential $d^2 f(v_0)$, without a factor $\frac{1}{2}$, and then one has to
prove that this expression is indeed bilinear;  similarly for higher
order differentials.

\begin{definition}\label{SJDef}
Assume that $f:U \to W$ is of class $\CC^{\langle k \rangle}$, and let $v_0 \in U$. 
For any $\sss \in \K^{k+1}$ we define the {\em simplicial $\sss$-extension of $f$} by
$$
\SJ^{(\sss)} f:  \SJ^{(\sss)} U \to W^{k+1}, \quad
\vvv \mapsto \begin{pmatrix}
f(v_0) \cr
f^{\langle 1 \rangle}(v_0,v_1;s_0,s_1) \cr
\vdots \cr
f^{\langle k \rangle}(v_0,\ldots,v_k;s_0,\ldots,s_k)\cr
\end{pmatrix} 
$$
where
$$
\SJ^{(\sss)} U := \big\{ \vvv \in V^{k+1} | \,
v_0 \in U, \quad \forall i=1,\ldots,k : \,
v_0+\sum_{j=1}^i \prod_{\ell=1}^j (s_i - s_\ell) v_j \in U \big\} 
$$
(this set is open in $V^{k+1}$).
For $\sss = (0,\ldots,0)$, the map 
$$
\SJ^k f:= \SJ^{(0,\ldots,0)} f : \, U \times V^k \to W^{k+1}
$$ 
is called the 
{\em simplicial $k$-jet of $f$}.
\end{definition}

\begin{theorem}\label{SimplicialChainRule}
{\rm (Chain rule)} 
The simplicial $\sss$-extension is a covariant functor:
if $f:U \to W$ and $g:U' \to W'$ are of
class $\CC^{\langle k \rangle}$ and such that $f(U) \subset U'$, then
$g \circ f$ is of class $\CC^{\langle k \rangle}$ and, 
for all $\sss \in \K^{k+1}$,
$$
\SJ^{(\sss)} (g \circ f)=  \SJ^{(\sss)} (g) \circ  \SJ^{(\sss)} (f) .
$$
The identity map $\id_U$ is of class $\CC^{\langle k \rangle}$ and satisfies
$\SJ^{(\sss)}  (\id_U) = \id_{ \SJ^{(\sss)} (U)}$.
In particular, for $\sss = 0$, we see that the simplicial $k$-jet defines a covariant functor.
\end{theorem}

\begin{proof}
First, we prove the claims for non-singular $\sss$, i.e.,
 $s_i - s_j \in \K^\times$; by density of $\K^\times$ in $\K$ and by continuity
 of terms on both sides of the equation, it will then hold for all $\sss$.
 
By induction, we show that
$\id^{\langle k \rangle}(\vvv;\sss)=v_k$. 
Indeed, for $k=0$ and $k=1$ it follows directly from the definitions,
and using  the recursion formula (Lemma \ref{RecursionLemma})
$$
\id^{<k+1>}(\vvv;\sss)=
\frac{v_k - (v_k + (s_{k+1} - s_k) v_{k+1})}{s_k-s_{k+1}} =v_{k+1} ,
$$
whence $\SJ^{(\sss)}  (\id_U) = \id_{ \SJ^{(\sss)} (U)}$.
 
Next we  show that, for non-singular $\sss$, $\SJ^{(\sss)} f$
is (linearly) conjugate to the direct product
$\times^{k+1} f: \times^{k+1}U \to \times^{k+1}W$.
In order to prove this, define the linear operator
\begin{equation} \label{M-matrix}
M_\sss : V^{k+1} \to W^{k+1}, \quad
\vvv \mapsto \bigl( v_0 + (s_i - s_0)v_1 + \ldots + \prod_{j=0}^i (s_i -s_j) v_i \bigr)_{i=0,\ldots,k}
\end{equation}
which may be identified with the  invertible lower
triangular $(k+1) \times (k+1)$-matrix
\begin{equation} \label{M-matrix'}
M_\sss = \begin{pmatrix}
1 &          &  &    &  \cr
1 & s_1 - s_0 &  &    &   \cr
1 & s_2 - s_0 & (s_2 - s_1)(s_2-s_0) &  & \cr
  &            &                     & & \cr
1 & s_k - s_0 & (s_k - s_1)(s_k-s_0)& \ldots & \prod_{i<k} (s_k - s_i) \cr
\end{pmatrix}
\end{equation}
Using this notation, the ``limited expansion'' from the
preceding theorem can be restated as follows:
for $i=0,\ldots,k$,
$$
f\bigl( (M_\sss \vvv)_i \bigr) = \bigl( M_\sss (\SJ^{(\sss)}f(\vvv)) \bigr)_i,
$$
that is, 
\begin{equation}
(\times^{k+1} f) \circ M_\sss  =   M_\sss \circ \SJ^{(\sss)} f  
\end{equation}
where $\times^{k+1} f: U^{k+1} \to W^{k+1}$ is simply the $k+1$-fold direct product of
$f$ with itself. 
The operator $M_\sss$ is invertible (since so is its ``matrix'');
let $N_\sss := (M_\sss)\inv$, so that
\begin{equation}\label{simplconj}
N_\sss \circ (\times^{k+1} f) \circ M_\sss = \SJ^{(\sss)} f . 
\end{equation}
In other words, the operator $\SJ^{(\sss)}$ is linearly conjugate to the direct product
functor and hence is itself a (covariant) functor:
for $f$ and $g$ as in the theorem, 
\begin{eqnarray*}
 \SJ^{(\sss)}(g \circ f) & = & N_\sss \circ (\times^{k+1} (g \circ f)) \circ M_\sss 
\cr
&=& N_\sss \circ (\times^{k+1} g) \circ (\times^{k+1} f) \circ M_\sss 
\cr
&=& N_\sss \circ (\times^{k+1} g) \circ M_\sss \circ  N_\sss \circ  (\times^{k+1} f) \circ M_\sss 
\cr
&=& 
\SJ^{(\sss)} g \circ \SJ^{(\sss)} f  \, .
\end{eqnarray*}
As explained above, by continuity and density the result follows for all $\sss$.
\end{proof}

One should regard $M_\sss$ as a ``change of variables'', which is bijective as long as
$\sss$ is non-singular, and then serves to ``trivialize'' the whole situation.
However, as soon as $\sss$ becomes singular, the change of variables is no longer
bijective, leading to the non-trivial structure of differential calculus. 
Nevertheless, certain features of the ``trivial'' situation survive, among them functoriality.
The promised ``explicit formula'' for $N_\sss$ with non-singular $\sss$ can be derived
easily from the explicit formula of the simplicial difference quotients: it is the linear map
$$
N_\sss:W^{k+1} \to V^{k+1}, \quad {\bf w} \mapsto
(N_\sss {\bf w})_i = 
\frac{w_0}
{\prod_{m=1,\ldots,k}(s_0-s_m)}
+ 
\sum_{j=1}^i
\frac{w_j}
{\prod_{m=0,\ldots,j \atop m \not= j}(s_j-s_m)} 
$$
which can be identified with a lower triangular matrix of the type
$$
N =
\begin{pmatrix}
1 &  & & & \cr
( s_0 - s_1)^{-1} &( s_1 - s_0)^{-1} & & & \cr
 ((s_0 - s_1)(s_0-s_2))^{-1}& ((s_1 - s_2)(s_1-s_0))^{-1}&
 ((s_2 - s_1)(s_2-s_0))^{-1} & & \cr
\ldots & & \ldots  & & \ddots  \cr
\end{pmatrix}
$$
Of course, one may check by a direct computation that the inverse matrix of $M_\sss$
is indeed given by such a formula.
Finally, we point out that, in the situation of Corollary \ref{DiffCor}, for $\sss = {\mathbf 0}$,
 the Chain Rule can be written out explicitly and then corresponds to the 
  ``F\'aa di Bruno formula'' (cf.\ \cite{Be08}, 8.7).
 
 \begin{corollary}\label{LimitedCorollary}
 A map $f:U \to W$ is of class $\CC^{\langle k \rangle}$ if and only if, for $j=1,\ldots,k$, there exist continuous
 maps $f^{\langle j \rangle}$ such that the ``limited expansion'' from Theorem \ref{LimitedTheorem}
 holds on the extended domain.
 \end{corollary}
 
 \begin{proof}
 One direction has been proved in Theorem \ref{LimitedTheorem}.
 As to the converse, the arguments given above show that the maps
 $f^{\langle j \rangle}$ are necessarily given by Equation (\ref{simplconj}), whence they
 indeed are continuous extensions of the simplicial difference quotient maps.
 \end{proof}
 
 Finally, in the next chapter we will need that the simplicial $\sss$-extension functors commute with direct products:
 
 \begin{lemma}
If $f$ and $g$ are of class
 $\CC^{\langle k \rangle}$, then so is $g \times f$, and
$$
\SJ^{(\sss)} (g \times f)=  \SJ^{(\sss)} (g) \times  \SJ^{(\sss)} (f) .
$$
It follows that $\SJ^{(\sss)}$ also is compatible with diagonal imbeddings
$\Delta:x \mapsto (x,x)$.
\end{lemma}

\begin{proof}
This is a rather a notational convention, meaning that we group together 
 terms coming from $f$ and those coming from $g$.
 \end{proof}

\section{The ring theoretic point of view}

In this chapter we are going to explain that the functors arising in higher order
difference- and differential calculus can all be understood as certain
{\em functors of scalar extension}. The basic remark is very simple:
whenever we have a covariant functor $F$ commuting with direct products,
applying $F$ to the base ring $\K$ yields a new ring $F\K$, and in the given
context, $F$ can then be interpreted as the functor of scalar extension by $F\K$.
Differential calculus corresponds to a ``contraction'' of the ring $F$: as the
parameter $\sss$ becomes singular, the ring $F=F_\sss$ tends to a
ring $F_0$ that is less rigid, hence allows for more symmetries and a richer
invariant theory. 

\subsection{First order difference calculus and quadratic scalar extension}

Recall from Definition \ref{cubicDef}  the functor $\hat T^{(t)}$, which is equivalent
to the functor $\SJ^{(0,t)}$ from Definition \ref{SJDef}.

\begin{lemma}
Let $(\K,a,m)$ be the base ring with addition map $a:\K\times \K \to \K$ and product map 
$m: \K\times \K \to \K$,
and let $F =\hat T^{(t)}$ be the extended tangent functor with fixed parameter
$t \in \K$.
Then $a$ and $m$ are cubically and simplicially smooth, and
$(F \K,Fa,Fm)$ is again a ring, which is isomorphic to the truncated polynomial
ring $\K[X]/(X^2 - tX)$.
\end{lemma}

\begin{proof}
Since $a$ and $m$ are linear (resp.\ bilinear), they are smooth, and
as a $\K$-module, $F\K = \K \times \K$. 
For invertible $t$, we get $\hat m^{(t)}$ by writing the difference quotient
\begin{eqnarray*}
(x_0,x_1)  \cdot (y_0,y_1)& =& 
\bigl(x_0y_0 \, , \,
{(x_0 + t x_1)(y_0 + t y_1) - x_0 y_0 \over t} \bigr) \cr
&  = & \bigl( x_0y_0 \, , \,  x_0y_1 + x_1y_0 + t x_1 y_1 \bigr) . 
\end{eqnarray*}
In a similar way, we see that the sum in this ring
 is just the usual sum in $\K^2$. Hence
as a ring, we get $\K \oplus \omega \K$ with relation $\omega^2 = t \omega$.
It can also be described as the truncated polynomial ring $\K[X]/(X^2 - t X)$.
Again by density, these statements remain true for non-invertible scalars $t$,
and in particular for $t=0$ we obtain the  {\em tangent ring} $T\K$,
which is nothing but the ring of  {\em dual numbers over $\K$}, $\K[X]/(X^2)=
\K \oplus \eps \K$, $\eps^2 =0$.
\end{proof}

\begin{theorem}\label{scalextTh}
Assume $f:U \to W$ is $\CC^{[2]}$ over $\K$.
Then $\hat T^{(t)} f$ is $\CC^{[1]}$ over the ring $\K[X]/(X^2 - tX)$.
\end{theorem}

\begin{proof}
The proof of the special case $t=0$ (\cite{Be08}, Theorem 6.3)
can be applied literally; it uses only the fact that
$\hat T^{(t)}$ is a covariant functor.  
\end{proof}

\nin To first order, the cubic and simplicial calculi coincide, and hence 
the following is a restatement of the preceding theorem:

\begin{theorem}
Assume $f:U \to W$ is $\CC^{<2>}$ over $\K$.
Then $\SJ^{(s_0,s_1)}f$ is $\CC^{\langle 1 \rangle}$ over the ring $\K[X]/(X^2-(s_1-s_0)X)$.
\end{theorem}

We add a few remarks on the structure of the ring $\K_t:= \K[X]/(X^2 - tX)$.
There is a well-defined projection
$$
\pi:\K_t \to \K, \quad [P(X)] \mapsto P(0)
$$
which splits via the natural map $\K \to \K_t$, $r \mapsto [r]$ (inclusion of constant
polynomials).
The kernel of the projection is isomorphic to $\K$ with product
$(a,b) \mapsto atb$; 
if $t$ is invertible, the kernel is isomorphic to $\K$ as a ring, and then
$\K_t$ is isomorphic to the direct product of rings $\K \times \K$.
If $t$ is nilpotent, the kernel is a nilpotent $\K$-algebra,
and if $t = 0$,  the kernel carries the zero product.

One may describe any element $z = a + \omega b \in \K_t$ by the $2 \times 2$-matrix 
representing the linear map
left translation by $z$. Hence, we are led to define $\tr(z):=2a+tb$ , $\det(z):=a^2 + tab$ and
$\overline z := a +bt - \omega b$. Then every $z$ satisfies the relation
$z^2 + \tr(z)z+\det(z)=0$, and $z$ is invertible iff $\det(z)$ is invertible in $\K$, in which
case $z\inv = \frac{\overline z}{\overline z \, z}=\frac{\overline z}{\det (z)}$. 

The automorphism group $\Aut_\K(\K_t)$ becomes richer as $t$ becomes singular.
If $t$ is invertible, $\K_t$ is isomorphic to $\K\times \K$, and the only non-trivial
$\K$-linear automorphism is the {\em exchange automorphism} exchanging both copies of $\K$.
Let us describe this automorphism in a more geometric way, that shows how this
automorphism survives also for singular $t$.
In general,
there are automorphisms arising from the affine group of $\K$ which acts on the
polynomial ring $\K[X]$; such automorphisms define automorphisms of $\K_t$ if
they preserve the ideal $(X^2 -tX)$ by which we take the quotient; and this ideal
is preserved if the affine map of $\K$ preserves the set of zeroes $\{ 0 , t \}$ of the ideal.
Thus,
if $t$ is invertible, the  exchange automorphism is induced by the affine map
exchanging the two roots (acting on polynomials  by
$[a+bX] \mapsto [a+b(t-X)]$, hence this is also the map $z \mapsto \overline z$
described above);
for $t=0$, there are more such automorphisms
since all dilations preserve the zero set $\{ 0 \}$, and hence we have a one-parameter
family of automorphisms, given by $[P(X)] \mapsto [P(rX)]$ with $r \in \K^\times$.

\subsection{Higher order cubic calculus and iterated scalar extensions}

In differential geometry, the iterated tangent functors $T^k = T \circ \ldots \circ T$
play an important role (see \cite{Be08}, \cite{White}). In a similar way, we may 
compose the scalar extension functors from the preceding section:
fixing $t'=s+s'X_1 \in \K':=\K_t = \K[X_1]/(X_1^2 - tX_1)$, we consider the iterated scalar extension
\begin{eqnarray*}
\A &:= &\K_{t'}' =  (\K_t)_{t'} = \bigl( \K[X_1]/(X_1^2 - t X_1) \bigr) [X_2]/
(X_2^2 - s' X_1 X_2 - s X_2)
\cr
&=& \K[X_1,X_2]/ \bigl( (X_1^2 - t X_1) ,
(X_2^2 - s' X_1 X_2 - s X_2) \bigr) 
\end{eqnarray*}
and applying Theorem \ref{scalextTh} twice, the second order functor $T^{(t,s,s')}$ 
is seen to be functor of scalar extension from $\K$ to $\A$.
More systematically,  we now construct a sequence
$\A_0,\A_1,\ldots$ of $\K$-algebras and of scalar extension functors, extending
the base ring from $\K$ to $\A_k$. 
At each step we have a quadratic ring extension, so that the dimension over $\K$
will double in each step. That is, we will obtain a canonical identification
$\A_k = \K^{2^k}$ as $\K$-modules.

\begin{definition}
Let $I=\{ 1,\ldots ,n \}$ be the standard $n$-element set and fix a family
$\ttt := (t_J)_{J \subset I}$ of elements $t_J \in \K$.
If $J= \{ i_1, \ldots, i_k \} $, then instead of
$t_J$ we write also $t_{i_1,\ldots,i_k}$, and 
we write $X_J:=X_{i_1} \cdots X_{i_k}$ for a product of indeterminates. 
Let $\A_k:=\A_k^{(\ttt)}$ be the $\K$-algebra 
$$
\A_k := \K[X_1,\ldots,X_k]/ R_k
$$
where $R_k=R^{(\ttt)}_k$ is the ideal generated by the polynomials (depending on $\ttt$)
\begin{eqnarray*}
P_1(X_1,\ldots,X_k) & = & X_1^2 - t_1 X_1 \cr
P_2(X_1,\ldots,X_k) & = &X_2^2 - t_2X_2 - t_{1,2} X_{1}X_2 \cr
P_3(X_1,\ldots,X_k) & = &X_3^2 - t_3 X_3 - t_{1,3} X_1 X_3 - t_{2,3} X_2X_3 - t_{1,2,3} X_1X_2X_3 \cr
& \vdots & \cr
P_k(X_1,\ldots,X_k) & = &X_k^2 - \sum_{J\subset \{ 1,\ldots,k-1 \} } t_{J \cup \{ k \} } X_J X_k \, .
\end{eqnarray*}
\end{definition}

\begin{lemma}\label{RingRecursionLemma}
The algebra $\A_k$ is a quadratic ring extension of $\A_{k-1}$. 
More precisely, $\A_k$ is a free $\K$-module having dimension $2^k$,
with canonical basis the classes of the polynomials
$X_J$ with $J \subset \{ 1,\ldots, k \}$, and as a ring,
$$
\A_{k}=\A_{k-1}[X_k]/ 
(X_k^2 - \ttt' \cdot X_k)
$$
where $\ttt' := (t_J)_{J\subset \{ 1,\ldots, k \}, k \in J}$ is identified with an element
of $\A_{k-1} = \K^{2^{k-1}}$ by mapping any
$J \subset \{ 1,\ldots, k \}$ such that $k \in J$ to the set $J \setminus \{ k \}$.
\end{lemma}

\begin{proof}
For $k=1$ the claim is obviously true. For general $k$,
the lemma translates merely the fact that, under the inclusion
$\K[X_1,\ldots,X_{k-1}] \subset \K[X_1,\ldots,X_k]$,
the ideal $R_k$ is generated by $R_{k-1}$ together with the polynomial $P_k$,
so that 
$$
\K[X_1,\ldots,X_{k}]/R_k =  \big( \K[X_1,\ldots,X_{k-1}] /R_{k-1} \big) / (P_k),
$$
and $P_k$ is a quadratic polynomial of $X_k$ if all variables except the last are
frozen.
\end{proof}

Combinatorial formulas for the ``structure constants'' $\Gamma^{JK}_L(\ttt) \in \K$, defined
by 
$$
X_J \cdot X_K \equiv \sum_{L \subset \{ 1,\ldots, k \} } \Gamma^{JK}_L(\ttt) X_L \, ,
$$
are fairly complicated. It is quite easy to see that $\Gamma^{JK}_L(\ttt)=0$
unless $(J \cup K) \subset L \subset \{ 1,\ldots, \max(J,K) \}$, and that, if $J \cap K = \emptyset$,
then 
$X_J \cdot X_K = X_{J \cup K}$.
The general case is illustrated by relations of the form 
$$
X_2 \cdot X_{\{ 1,2 \}} = X_1 X_2^2 \equiv  X_1 (t_2 X_2 + t_{1,2} X_1 X_2) =
(t_2 + t_1 t_{1,2}) X_{\{ 1,2 \}},
$$
$$
X_2 \cdot X_{\{ 2,3 \} } = X_2^2 X_3 \equiv (t_2 X_2 + t_{1,2} X_1 X_2)X_3 =
 t_2 X_{ \{ 2,3 \} } +
t_2 t_{1,2} X_{\{ 1,2,3 \} }.
$$
Of course, for special choices of $\ttt$ the structure may become much simpler;
this is in particular the case for $\ttt = 0$, where we get the higher order tangent ring
$T^k \K$. Similar remarks hold concerning inversion in $\A_k$.

\begin{theorem}\label{scalextTh3}
Assume $f:U \to W$ is $\CC^{[k+m]}$ over $\K$ and let
$\ttt \in \K^{2^k-1}$.
Then $T^{(\ttt)}f$ is $\CC^{[m]}$ over the 
algebra $\A_k^{(\ttt)}$.
\end{theorem}

\begin{proof}
The result follows by induction from Theorem \ref{scalextTh}
since, by the lemma, the inductive definition of the rings
$\A_{k}$ corresponds to the inductive definition of the functors $T^{(\ttt)}$.
\end{proof}

Let us add some remarks on the structure of the rings $\A_k$, and in particular on their
automorphisms. For simplicity, let us consider the case $k=2$. 
There are surjective ring homomorphisms
$$
\A_2 \to \A_1 \to \K, \quad P(X_1,X_2)\mapsto P(X_1,0) \mapsto P(0,0)
$$
 which admit sections. 
Note that $P(X_1,X_2) \mapsto P(0,X_2)$ does not pass to a well-defined homomorphism
on $\A_2$; however, this is the case if $t_{1,2}=0$ (in this case, the rings
$\A^{(t_1,t_2,0)}$ and $\A^{(t_2,t_1,0)}$ are isomorphic).

As in the first order case, the automorphism group becomes richer as $\ttt$ tends to singular
values:
for  non-singular $\ttt$, iterating the ring isomorphism $\A_1 \cong \K\times \K$,
we have $\A_2 \cong \K^4$, and hence the permutation group $\Sigma_4$ acts by
automorphisms of $\A_2$. Two of these automorphisms are 
the  two commuting exchange automorphisms, coming in each step from
the quadratic  extension $\K \subset \A_1 \subset \A_2$. 
The others do not seem  to have a simple geometric description.

On the other hand, ``geometric'' automorphisms come from the affine group of $\K^2$,
acting on $\K[X_1,X_2]$ in the usual way:
namely,  if $\K$ has no zero-divisors,  the equations
$X_1(X_1- t_1)= 0 $,
$X_2(X_2 - t_{1,2} X_1 - t_2 )=0$
define two pairs of lines forming a trapezoid in $\K^2$. An affine transformation of $\K^2$
preserving this figure gives rise to an automorpism of $\A_2$.
In the generic case, there is exactly one non-trivial such map (it is of order $2$).

If $t_{1,2}=0$ and $t_1=t_2$, then the trapezoid becomes a square, and we obviously have a
new symmetry exchanging both axes (the ``flip''): this symmetry is precisely the one giving rise
to Schwarz's lemma (see its proof in \cite{BGN04}, Lemma 4.6).
If moreover $t_1 = t_2 =0$, then we are in the case of the ring $TT\K$, and the figure 
degenerates to two perpendicular lines --
in particular, this figure is preserved by all $2 \times 2$-diagonal matrices, and by their
composition with the flip (if $\K=\R$, this gives the full description of the automorphism group,
see \cite{KMS}, p.\ 320).  

If $t_{1,2}=1$ and $t_1=t_2=0$, the figure degenerates to three concurrent lines, and all
multiples of the identity on $\K^2$ give rise to endomorphisms of $\A_2$.

\subsection{Simplicial calculus and simplicial ring extensions}

\begin{theorem}\label{SimpScalExtTheorem}
Fix $\sss = (s_0,\ldots,s_k) \in \K^{k+1}$ and assume $s_0=0$ (otherwise replace
$s_i$ by $s_i-s_0$).
 Then the simplicial $\sss$-extension functor
from Theorem \ref{SimplicialChainRule}  is the functor of scalar extension from $\K$ to the ring
$$
\B_k:=\B^\sss_k:=\K[X]/(X(X-s_1))\ldots (X-s_k)),
$$
that is, if $f:U \to W$ is $\CC^{<k+m>} $ over $\K$, then
$\SJ^{(\sss)} f$ is $\CC^{<m>}$ over $\B_k$.
In particular, if $s_i=0$ for all $i$, we get the jet functor of
scalar extension from $\K$ to $\K[X]/(X^{k+1})$.
\end{theorem}

\begin{proof} 
The proof from Theorem \ref{scalextTh} carries over to the
present situation, {\it mutatis mutandis}: let $F$ be a functor of the type in question (covariant
and preserving direct products). 
Recall from Corollary \ref{LimitedCorollary} that $f$ is $\CC^{<m>}$ if and only if
there exists a continuous map $(\vvv,\ttt) \mapsto g_\ttt(\vvv)$ such that
$$
 \times^{m+1} f \circ M_\ttt = M_\ttt \circ g_\ttt \,  .
$$
If $f$ is $\CC^{<m+k>}$, then $g$ is actually of class $\CC^{\langle k \rangle}$ and hence we
can apply the functor $F = \SJ^{(\sss)}$
to this relation; we obtain a relation of the same kind, where $f,g_\ttt$ are replaced by
$Ff, Fg_\ttt$, and $\K$ by $F\K$, and $V,W$ by their scalar extensions $V_{F\K}$,
$W_{F\K}$. This proves that $Ff$ is $\CC^{<m>}$ over $F\K$.
It remains to determine the ring $F\K$. This is the content of  the following lemma:

\begin{lemma}\label{SimpScalExtLemma}
The rings $\B^\sss_k$ and $\SJ^{(\sss)}\K$ are canonically isomorphic. More precisely,
if $b_0,\ldots,b_k$ denotes the standard basis in $\SJ^{(\sss)} \K = \K^{k+1}$ and
$c_0,\ldots,c_k$ the basis of $\B_k$ given by (the classes of) the polynomials
$$
c_j (X) = X (X-s_1) \cdot \ldots \cdot (X-s_j),
$$
then $\B_k^\sss \to \SJ^{(\sss)}\K$, $c_j \mapsto b_j$ is a ring isomorphism.
In particular, for $\sss=0$, the standard bases of these rings correspond to each other.
\end{lemma}

\nin {\em Proof of the Lemma.}
Once again it suffices to prove the claim for non-singular $\sss$. Indeed, the map
in question is always a  $\K$-linear bijection. Hence, if we have shown that it is a ring
isomorphism for non-singular $\sss$, then, 
since the products on both sides depend continuously on $\sss$, by density of the non-singular
elements this map will be a ring isomorphism for all $\sss$. 

For non-singular $\sss$, since $X-s_i$ and $X-s_j$ are then coprime for
$i \not= j$, by the Chinese Remainder Theorem,
$\B_k^\sss$ is uniquely isomorphic to the direct product of rings
$\prod_{i=0}^k \K[X]/(X-s_i) = \K^{k+1}$.
Thus there is a unique $\K$-basis $e_0,\ldots,e_k$ of $\B_k^\sss$ such that
$e_i \cdot e_j = \delta_{ij} e_i$. In fact, $e_i$ is the class of the polynomial $E_i(X)$ of
degree $k$ satisfying
$$
\forall j=0,\ldots,k: \quad \quad 
E_i(s_j) = \delta_{ij} .
$$
These polynomials are determined as follows:
let $A := A_\sss := (a_{ij})_{i,j=0,\ldots,k}$ be the change of basis  matrix, defined by
$
c_j = \sum_{i=0}^k a_{ij} e_i \, .
$
It follows that
$$
a_{ij} = \sum_{n=0}^k a_{jn} E_n(s_i) = c_j(s_i)=
s_i(s_i-s_1) \cdot \ldots \cdot (s_i - s_j) .
$$
Note that these are exactly the coefficients of the 
 matrix $M_\sss$ given by Equations (\ref{M-matrix}), resp.\
(\ref{M-matrix'}), whence $A_\sss = M_\sss$.

On the other hand,  as seen in the proof of Theorem
\ref{SimplicialChainRule}, the simplicial $\sss$-extension of the product map
$m:\K \times \K \to \K$ is conjugate to a direct product $\times^{k+1} m$ via
$$
\SJ^{(\sss)} m = N_\sss \circ \times^{k+1} m \circ M_\sss
$$
where $N_\sss = (M_\sss)\inv$.
Therefore the new basis $f_j:= N_\sss(b_j)$ in $\K^{k+1}=\SJ^{(\sss)}\K$
is characterized by the idempotent relations
$f_j \cdot f_i = \delta_{ij} f_j$.
Since $A_\sss = M_\sss$, the bases $e_j$ and $f_j$ correspond to each other
under the bijection from the lemma, and they satisfy the same
multiplication table.
This proves the lemma and the theorem for non-singular $\sss$
and hence for all $\sss$.
\end{proof}

We add a few remarks on the structure of the ring:
there are projections $\B_{k+1} \to \B_k$, hence by composition $\B_k \to \B_j$ for
$j \leq k$, 
but  these projections do not have
a section, except for  $j=0$. 
As to the automorphism group, 
if $\sss$ is non-singular, there is of course an action of the symmetric group
on $\B_k \cong \K^{k+1}$, permuting the roots $s_i$.
This action degenerates for singular $\sss$, and 
for $\sss = 0$ survives by a sign:
namely, for $\sss=0$, every dilation of $\K$ acts on the polynomial algebra
$\K[X]$, and this action descends to $\B_k$.

\subsection{Embedding of simplicial ring extensions into cubic ones}

Recall that, if $f$ is $\CC^{[k]}$, then $f$ is $\CC^{\langle k \rangle}$, and the $\sss$-extended simplicial
divided differences can be embedded into the cubic $\ttt$-extension (Theorem \ref{cubictosimplicialTh}).
This means that, on the ring level, the rings $\B^{(\sss)}_k$ can be embedded into the
algebras $\A^{(\ttt)}_k$. In the following, we prove a purely algebraic version of this
result:

\begin{theorem}\label{RingEmbeddingTheorem}
Fix $\sss \in \K^{k+1}$, and  assume that $s_0 =0$.
Let $\ttt \in \K^{2^k - 1}$ be such that for all $i=0,\ldots,k$,
$$
t_{ \{ i \} }= t_i = s_{k-i} - s_{k-i-1}, \quad
t_{\{ i,i+1 \}}= t_{i,i+1} = 1, \quad
t_J = 0  \mbox{ else.  } 
$$
Then the subring $\langle X_k \rangle$ of $\A^{(\ttt)}_k$ generated by the class of the polynomial
$X_k$ is isomorphic to $\B_k^\sss$.
\end{theorem}

\begin{proof}
By choice of $\ttt$, $\A_k^{\ttt}$ is the polynomial ring $\K[X_1,\ldots,X_k]$, quotiented by
the relations
$$
X_1^2 \equiv t_1 X_1, \quad \forall j=2,\ldots,k: \, X_i^2 \equiv X_{i-1} X_i + t_i X_i \, .
$$
Let $\B \subset \A_k^{\ttt}$ be the $\K$-submodule 
$$
\B:= \K \oplus \K X_k \oplus \K X_k X_{k-1}
\oplus \K  X_k X_{k-1}X_{k-2} \oplus \ldots \oplus \K X_k X_{k-1} \cdots X_1 \, .
$$
We claim that $\langle X_k \rangle = \B$.
Indeed, by an easy induction it follows from the relations written above that, for
all $j,\ell \in \N$ there exist constants $c_1,\ldots,c_\ell \in \K$ such that
$$
X_j^\ell \equiv  X_j  X_{j-1} \cdots X_{j - \ell +1} +
c_1 X_j  X_{j-1} \cdots X_{j - \ell} + \ldots +
c_{\ell - 1} X_j,
$$
whence $X_k^\ell \in \B$, whence $\langle X_k \rangle \subset \B$.
On the other hand,
$X_k X_{k-1} \equiv  X_k^2 - t_k X_k $ belongs to $\langle X_k \rangle$, hence also
$X_k X_{k-1} X_{k-2} \equiv  X_k^3 - c_1 X_k X_{k-1} - c_2 X_k$ belongs to 
$\langle X_k \rangle$, and so on, whence the other inclusion and hence
$\langle X_k \rangle = \B$.

Let $(P)$ be the kernel of the surjective homomorphism
$\phi: \K[X] \to \bB$ sending $X$ to $X_k$, so that
$\bB \cong \K[X]/(P)$. Now we show that this establishes an isomorphism of
$\B$ with $\bB_k^\sss$. Again, by density, it will suffice to prove this for non-singular $\sss$, since
the products on both sides depend continuously on $\sss$.
For reasons
of dimension, 
$P$ is a polynomial of degree $k$. We show by induction:
{\em the polynomial $P$ has simple roots in $\K$, equal to
$0,s_1,\ldots,s_k$, hence is proportional to
$X(X-s_1)\cdot \ldots \cdot (X-s_k)$.}
Indeed, for $k=1$ we have $X_1 (X_1 - t_1) = 0$, hence $0$ and $t_1$ are roots of $P$,
and they are simple since $P$ is of degree two and  $t_1$ is invertible. 
Assume the claim proved at rank $k-1$, i.e.,
$$
X_{k-1}(X_{k-1} - s_1) \cdot \ldots \cdot (X_{k-1} - s_{k-1}) \equiv 0 \, .
$$
We multiply by $X_k$, and note that (using the defining relations of $\A_k^\ttt$)
$$
X_k(X_{k-1} - s_j) \equiv  X_k^2 - t_k X_k - s_j X_k =
(X_k - (t_k + s_j))X_k \, ,
$$
so that we get
\begin{eqnarray*}
0 & \equiv & X_k X_{k-1}(X_{k-1} - s_1) \cdot \ldots \cdot (X_{k-1} - s_{k-1}) 
\cr & \equiv &
(X_k - t_k) X_k (X_{k-1} - s_1) \cdot \ldots \cdot (X_{k-1} - s_{k-1}) 
\cr & \equiv & \ldots \cr
& \equiv &
 (X_k - t_k) (X_k - (t_k +s_1)) \cdot \ldots \cdot (X_k - (t_k +s_{k-1})) X_k \, .
\end{eqnarray*}
Thus, at rank $k$, $P$ has necessarily as roots
$0,t_k,t_k+s_1,\ldots,t_k + s_{k-1}$.
\end{proof}

The embedding $\B_k^{(\sss)} \subset \A_k^{(\ttt(\sss))}$ just constructed coincides in fact with
the embedding of $f^{\langle k \rangle}$ into $f^{[k]}$ constructed in Theorem \ref{cubictosimplicialTh},
which could have been used to give another (less algebraic) proof of the preceding
result.

\section{The Weil functors}

\subsection{Manifolds and bundles of class $\CC^{\langle k \rangle}$}
We define {\em manifolds of class $\CC^{[k]}$}  as in \cite{BGN04} or
in \cite{Be08}, Section 2. 
The definition of {\em manifolds of class $\CC^{\langle k \rangle}$} follows the same pattern.
 We briefly recall the  relevant definitions. 
Fix a topological $\K$-module $V$ as ``model space'', and let $M$ be a topological space
(which may be Hausdorff or not). 
We say that
 ${\mathcal A} = (\phi_i,U_i)_{i \in I}$ is an {\em atlas of $M$ (of class $\CC^{\langle k \rangle}$)}, 
if $I$ is some index set, $(U_i)_{i\in I}$ an open covering of $M$ and
$\phi_i:U_i \to V_i$ bijections with open sets $V_i \subset V$ such that the transition functions
$$
\phi_{ij}:V_{ji} \to V_{ij}, \quad v \mapsto \phi_i(\phi_j^{-1}(v)),
\quad \mbox{ where }
V_{ij} := \phi_i(U_i \cap U_j),
$$
are of class $\CC^{\langle k \rangle}$. On 
$$
S := \{ (i,x) \in I \times V  \mid \,x \in  V_i  \}
$$
define an equivalence relation $(i,x) \sim (j,y)$ iff $\phi_i\inv(x)=\phi_j\inv(x)$ iff 
$\phi_{ji}(x)=y$. Then $S/\sim \to M$, $[i,x] \mapsto \phi_i\inv(x)$ is a bijection, and
a set $Z \subset M$ is open if, and only if, for all $i \in I$, the set
$Z \cap U_i$ is open in $M$, if and only if, for all $i \in I$, the set 
$$
Z_i := \phi_i (Z \cap U_i) = \{ x \in V \mid \, [i,x] \in Z \}
$$
is open in $V$.  This can be rephrased by saying that the topology of $M$ is
recovered as the quotient topology of the canonical projection
$S \to M=S/\sim$, where $S \subset I \times V$ carries the topology
induced from the product $I \times V$, where $I$ carries the discrete topology.
As may be checked directly, all of these constructions admit a converse (see \cite{St51},
pp.\ 14 -- 15, where essentially the same construction is described in a slightly
different context). We summarize:

\begin{proposition}
A $\CC^{\langle k \rangle}$-manifold with atlas, indexed by $I$ and modelled on
 a topological $\K$-module $V$, is
equivalent to the following data:
a collection of open sets $V_{ij}\subset V$ such that
$V_i:=V_{ii}$ is non-empty, and a collection of $\CC^{\langle k \rangle}$-diffeomorphisms
$\phi_{ij}:V_{ji}\to V_{ij}$ such that $\phi_{ii}=\id_{V_i}$ and
$\phi_{ij}\circ \phi_{jk}=\phi_{ik}$ (on $V_{kj} \cap \phi_{jk}\inv(V_{ji})$);
the manifold is then given by $M=S/\sim$ with  equivalence relation and quotient
topology as described above; the atlas is given by
$U_i:= (\{ i \} \times V_i)/\sim$ and
$\phi_i:U_i \to V_i$, $[i,x] \mapsto x$.
\end{proposition}
 
For $M$ to  be Hausdorff it is necessary, but not sufficient that  the model space $V$ be Hausdorff.
One may always shrink chart domains since obviously the restriction of
a $\CC^{\langle k \rangle}$-map to a smaller open set is again $\CC^{\langle k \rangle}$;
however, if $\K$ is not a field, one has to be careful with unions of chart domains
(see \cite{Be08}, 2.4). 
One could assume that the atlas is maximal (in the usual sense), 
but this will not be important in the sequel. 

\begin{theorem}\label{FunctorTheorem}
Let $F$ be one of the functors $\hat T^{(\ttt)}$, resp.\ $\SJ^{(\sss)}$ 
(of degree $k$), let
$M$ be a manifold with atlas of class $\CC^{\langle \ell \rangle}$ (with $\ell \geq k$), and
retain notation from above.
Then the data $(FV,(F(V_{ij}),F(\phi_{ij}))_{i,j\in I})$ define a manifold $FM$ with atlas
$F{\mathcal A}$ which is of class   $\CC^{\langle \ell - k \rangle}$ over the ring $F\K$
(and hence also over $\K$).
There is a canonical projection $\pi:FM \to M$.
The construction is functorial in the category of manifolds with atlas.
\end{theorem}

\begin{proof}
By functoriality, the data  $(FV,(F(V_{ij}),F(\phi_{ij}))_{i,j\in I})$ satisfy again
the condition  of the preceding proposition, and hence define a manifold $FM$.
As shown in the preceding chapter, the functors $F$ admit a natural ``base projection''
$\pi:F (V_{ij}) \to V_{ij}$, i.e., $\pi \circ F(\phi_{ij}) = \phi_{ij} \circ \pi$,
and hence $\pi$ gives rise to a globally well-defined map $FM \to M$.
It is clear from the definition of morphisms  of manifolds, 
and from the functoriality of $F$ on the level
of open sets, that the construction is functorial.
Finally, again by results of the preceding chapter, $F(\phi_{ij})$ is smooth over the
ring $F\K$, hence $FM$ is a manifold not only over $\K$, but also over $F\K$.
\end{proof}

In general, not only the atlas $F {\mathcal A}$, but also the underlying set
 of the manifold $FM$ will depend on the atlas ${\mathcal A}$.
This is best understood by looking at the following example:

\begin{example}  
Let $F = \widehat T^{(1)}$, that is,
$$
FU = \{ (x,v) \mid \, x \in U, x+v \in U \}, \quad 
Ff(x,v)=(f(x),f(x+v) - f(x)) \, .
$$
As we have already remarked, this functor is conjugate to the direct product
functor, via the simple change of variables $(x,y):=(x,x+v)$. 
This change of variables is the chart formula of a well-defined embedding 
$\iota: FM \to M \times M$, $[i;x,v] \mapsto [i;x,x+v]$. 
In general, $\iota$ will  not be a bijection: 
the fiber over the point $p\in M$ is the set of all $q \in M$ such that, for some $i \in I$, the points
$p$ and $q$ belong to a common chart domain $U_i$; let us temporarily call this 
set the {\em star of $p$}.
If the atlas consists of ``small'' charts, then the star of $p$ may very well be a proper
subset of $M$. If we choose a maximal atlas, including ``very big'' charts,
then under quite general conditions the star of $p$ equals $M$ (in the Hausdorff case, e.g.,
over $\K=\R$, we may work with non-connected charts; in this case one might distinguish
between a ``connected star'' and a general one). 
For instance, for finite-dimensional real
projective spaces, Grassmannians and reductive Lie groups, the ``connected star'' will
be equal to $M$, and 
hence $FM = M \times M$ in these cases.
Note that, even when using a ``small'' atlas, functoriality implies that, e.g.,
if $G$ is a Lie group, then so is $FG$. One may think of the Lie
group $FG$ then as some ``open neighborhood group'' of the diagonal group $\Delta(G \times G)$.
\end{example}

\subsection{Locality} \label{sec:locality}
The example we have just discussed is of ``non-local'' nature: in the language of \cite{KMS},
it corresponds to a product-preserving functor that
stems from a formally-really algebra ($F \R \cong \R \times \R$, in the example) which is
not a Weil algebra. The ``non-locality'' is related to the dependence on the atlas. 
On the other hand, we know that
the fibers of the tangent bundle $TM$ and of its iterates $T^k M$ are independent of
the chosen atlas of $M$. This corresponds to ``locality'' of the tangent functor, and to the fact
that $T\K = \K[X]/(X^2)$ is a Weil algebra.

\begin{definition}
Let $F$ be one of the functors $\hat T^{(\ttt)}$, resp.\ $\SJ^{(\sss)}$, defined as in Chapter 1.
We say that $F$ is {\em local} if,
for an open set $U$ in the model space $V$,
$$
F U = U \times V^{2^k-1}, \quad \mbox{ respectively } \quad FU = U \times V^k  .
$$
\end{definition}

\begin{theorem}\label{WeilFunctorTheorem}
\begin{enumerate}[label=\roman*\emph{)},leftmargin=*]
\item
The functor $\SJ^{(\sss)}$ is local for $\sss =0$.
\item
The functor $\hat T^{(\ttt)}$ is local if $t_J =0$ whenever $J$ is of cardinality one
(i.e., if $t_i = 0$ for $i=1,\ldots,k$).
\end{enumerate}
\end{theorem}

\begin{proof} i)
Recall the definition of $\SJ^{(\sss)}U$ (Definition 1.9).
If $\sss=0$, it is obvious that only the condition $v_0 \in U$ remains, and all other
$v_i$ can be chosen arbitrarily, hence the functor $\SJ^{(0)}$ is local.

ii)
Let $F:=\hat T^{(\ttt)}$ and assume that $t_i=0$ for $i=1,\ldots,k$.
We prove by induction that $FU=U \times V^{2^k-1}$.
For $k=1$ and $t_1=0$, 
$$
FU = \{ (x,v) \vert \, x \in U, v \in V, x+t_1v \in U \} = U \times V
$$
(and this case corresponds to the tangent bundle).
For the inductive step, we use the recursion relation for $\A_k$ from Lemma \ref{RingRecursionLemma},
which gives the following recursion relation for the domains (notation as in the lemma, and
write $T_j U $ for $T^{(\ttt)}U$ if $\ttt \in \K^{2^j - 1}$)
$$
T_k U = \{ (x,w) \in T_{k-1} U \times V^{2^{k - 1}} \mid \,
x + \ttt' . w \in T_{k-1} U \} \,
$$
where the product $\ttt'.u$ is the action of the ring $\A_{k-1}$ on the scalar extension
$V_{\A^{k-1}} = V^{2^{k-1}}$. By induction,
$T^{k-1}U=U \times V^{2^{k-1}-1}$. Therefore, writing out the $2^{k-1}$ components of the
condition $x + \ttt'.w \in T_{k-1} U$, only the first component may add a non-trivial
condition (all other conditions mean that some vector lies in $V$, which is always true).
But this first condition is of the form 
$(\ttt')_1 w_1 \in U$, where $(\ttt')_1 = t_k =0$ by asumption, and hence, is also satisfied for
all $w_1 \in V$, hence any $w \in V^{2^{k-1}}$ fulfills the condition.
Moreover, again by induction, any $x \in U \times V^{2{k-1}-1}$ belongs to
$T_{k-1}U$; summarizing, $T_k U = U \times V^{2^k -1}$.
\end{proof}

\begin{theorem}\label{WeilFunctorTheoremBis}
Assume, in the situation of Theorem \ref{FunctorTheorem}, that $F$ is local
(i.e., $\sss$, resp.\ $\ttt$ are as in Theorem \ref{WeilFunctorTheorem}).
Then the bundle
$\pi:FM \to M$ is locally trivial with typical fiber $V^r$ where $r=2^k -1$, resp.\ $r=k$;
in particular,
 as a set and as a topological space, $FM$ does not depend on the atlas ${\mathcal A}$ chosen on $M$.
\end{theorem}

\begin{proof} For an element $p=[i,x] \in M$, let
$F_p M = \{ [i;x,\vvv] \mid \, (x,\vvv) \in F V_i \}$
be the fiber of $FM$ over $p$. By locality, the map
$F_p \phi\inv_i:V^r \to F_p M$, $\vvv \mapsto [i;x,v]$ is an isomorphism of $\K$-modules.
Thus the bundle $\pi:FM \to M$ is locally trivial with typical fiber $V^r$,
and this property does not depend on the atlas ${\mathcal A}$.
\end{proof}

\ssk
The functors from the preceding theorem are generalizations of the Weil functors from the
real theory, and the corresponding rings generalize the Weil algebras from \cite{KMS}.
We add some final remarks.

\ssk
1.
Comparing the functors $\SJ^{(\sss)}$ and $\hat T^{(\ttt)}$, the ``more efficient'' organization
of the simplicial functor corresponds to the fact it has just one ``local contraction'', whereas
the cubic functor admits many of them. This may lead to the conjecture that the generalized
divided differences also give the `correct' definition of a ``pointwise'' concept of differentiability:
one may say that {\em a map $f$ is $\CC^{\langle k \rangle}$ at a point $p \in U$}
if all limits 
 $\lim_{\sss\to 0,\vvv \to (p,0,\ldots,0)} f^{\rangle j\langle }(\vvv;\sss)$  exist.
The proofs of the chain rule and of the Taylor expansion then go through essentially without any
changes. 
On the cubic level, similar ``pointwise'' concepts can be defined, but appear to be less
natural since the limit condition must be formulated differently (for $t_i \to 0$ the limits
shall exist while the other components of $\ttt$ may remain arbitrary).

\ssk
2.
As mentioned in the introduction, it is an important topic for further work to adapt the approach
to differential geometry and Lie theory over general base rings from \cite{Be08} to this
the simplicial framework introduced here. 
As long as it is not clarified whether the converse of Theorem  \ref{cubictosimplicialTh}
holds, one should still work in the $\CC^{[\infty]}$-category since in this case we already
know that the simplicial jet bundles $\SJ^k M$ over $M$ will be {\em polynomial bundles}
(i.e., the transition functions are polynomial in the fibers); for the 
$\CC^{\langle \infty \rangle}$-category,
this would follow as a corollary from the conjectured converse of Theorem \ref{cubictosimplicialTh}.
Although we did not use partial derivatives in an explicit way, our
interpretation of jet bundles in \cite{Be08} followed common definitions; over $\R$,
or over any ring of characteristic zero, higher order tangent bundles $T^k M$ or
their symmetric parts $J^k M$ are indeed equivalent objects, so that one may work with either
of them. However, the difference between them is that the bundle projections 
$T^k M \to T^jM$ have canonical sections, whereas the bundles $J^k M \to J^j M$ do not.
An intrinsic, or ``simplicial'', theory of the bundles $J^k M$ should not use the
sections of the ambient $T^kM$; such a theory would then automatically be valid for the
bundles $\SJ^k M$, and hence be fully valid also in positive characteristic.
Analogous remarks apply to Lie theory, and in particular to the relation between Lie groups and Lie
algebras. We will discuss such topics in subsequent work.




\end{document}